\documentclass[review]{elsarticle}
\usepackage{amsmath,amsthm} 
\usepackage{amssymb}
\usepackage{mathtools}
\usepackage{relsize}
\usepackage{graphicx}
\usepackage{rotating}
\usepackage{pdflscape}
\newtheorem{theorem}{Theorem}[section]
\newtheorem{definition}[theorem]{Definition}
\newtheorem{corollary}[theorem]{Corollary}
\newtheorem{proposition}[theorem]{Proposition}
\newtheorem{lemma}[theorem]{Lemma}

\numberwithin{equation}{section}
\usepackage{lineno,hyperref}
\modulolinenumbers[5]

\journal{European Journal of Combinatorics}

\begin{document}

\begin{frontmatter}

\title{Enumerating path diagrams in connection with $q$-tangent and $q$-secant numbers}
\author{Anum Khalid}
\author{Thomas Prellberg}
\address{School of Mathematical Sciences, Queen Mary University of London, Mile End Road, London E1 4NS, United Kingdom}
\ead{t.prellberg@qmul.ac.uk}
\begin{abstract}
We enumerate height-restricted path diagrams associated with $q$-tangent and $q$-secant numbers by considering convergents of continued fractions, leading to expressions involving basic hypergeometric functions. Our work generalises some results by M. Josuat-Verg\'es for unrestricted path diagrams [European Journal of Combinatorics {\bf 31} (2010) 1892].
\end{abstract}

\end{frontmatter}
\section{Introduction and Statement of Results}
Much work has been done on the enumeration of non-crossing directed lattice paths in both the mathematics and the physics communities, see e.g.~\cite{walkers,dwalks}. The work here takes into account two paths given by a path diagram, i.e.~a Dyck path and a general directed path restrained to lie between the $x$-axis and this Dyck path. In particular, we shall consider Dyck paths restricted by height.

A Dyck path is a lattice path on $\mathbb{N}^2$ from $(0,0)$ to $(2n,0)$ consisting of $n$ steps in the northeast direction of the form $(1,1)$ and $n$ steps in the southeast direction of the form $(1,-1)$ such that the path never goes below the line $y=0$.  
We encode a Dyck path in terms of labelled steps where each step is indexed with the height of the point from where it starts. For example, the labelled path shown in Figure \ref{fig:dp_0} is encoded as $(a_0,b_1,a_0,a_1,a_2,b_3,b_2,a_1,b_2,b_1)$, where $a_i$ is a northeast step starting at height $i$ and $b_j$ is a southeast step starting at height $j$. So we can say that there is a set $X=\{a_0,a_1,a_2, \ldots\} \cup \{b_1,b_2,b_3,\ldots\}$, the elements of which, as an ordered finite sequence, are associated with a Dyck path. We consider path diagrams \cite{Flajolet2006992} which are represented by a Dyck path and the set of points under it subjected to some conditions expressed using the above encoding.
\begin{definition}[\cite{Flajolet2006992}]{\textit{Path Diagrams.}}
A system of path diagrams is defined by a possibility function
$$pos:X\rightarrow \mathbb{N}_0.$$
Path diagrams are  composed of the Dyck path $u=u_1 u_2 u_3 \ldots u_n$  where for $j=1,2,\ldots,n$ each $u_j \in X$, and the corresponding sequence of integers $s=s_1 s_2 s_3 \ldots s_n$ where for $i=1,2,\ldots,n$ each $0 \leq s_i \leq pos(u_j)$. We get $n$ points corresponding to a path of length $n$.
\end{definition} 
We consider two types of path diagrams. In the first case we consider all possible lattice points bounded by the $x$-axis and a Dyck path by using the possibility function
\begin{equation}
pos(a_j)=j, \quad pos(b_k)=k, \quad\text{for}\quad  j \geq0 \quad\text{and}\quad k \geq1.
\label{tangentfunc}
\end{equation}  
In the second case we restrict this set of points by excluding the points which are in contact with the Dyck path at a southeast step, leading to
\begin{equation}
pos(a_j)=j, \quad pos(b_k)=k-1, \quad\text{for}\quad  j \geq0 \quad\text{and}\quad k \geq1.
\label{secantfunc}
\end{equation}
These two possibility functions map labelled steps onto a set of integers. These integers can be visualised as column heights, and a path is then formed by joining the peaks of the columns. 

\begin{figure}[h!]
\begin{center}\includegraphics[width=.9\textwidth]{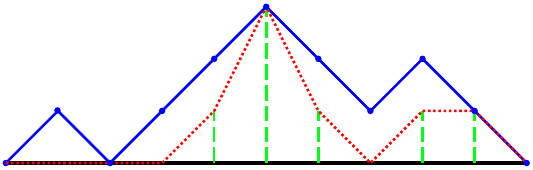}\end{center}
\caption{A Dyck path of half length $N=5$ (solid blue line), together with columns of heights formed by the integers $(0,0,0,0,1,3,1,0,1,1)$ (dashed green lines). The dotted red line combined with the Dyck path represents the path diagram.}
\label{fig:dp_0}
\end{figure}

Figure \ref{fig:dp_0} shows an example of one such path diagram given the Dyck path example used above. Columns of heights are formed by the sequence of integers $(0,0,0,0,1,3,1,0,1,1)$, with the associated path shown as a dotted line. When restricting the height of the Dyck path, we can interpret this as a model of two non crossing paths in a finite slit.  

Let $a^{(w)}_{N,m}$ be the number of path diagrams defined by the possibility function \eqref{tangentfunc} and $b^{(w)}_{N,m}$ be the number of path diagrams formed by the possibility function \eqref{secantfunc}, bounded by a Dyck path of length $2N$ in a slit of width $w$. Then we define the associated generating functions
\begin{equation}
G_w(t,q)=\sum_{N,m=0}^{\infty}a^{(w)}_{N,m} t^{2N} q^m
\label{1}
\end{equation}
and
\begin{equation}
G'_w(t,q)=\sum_{N,m=0}^{\infty}b^{(w)}_{N,m} t^{2N} q^m,
\label{2}
\end{equation}
with the variable $q$ conjugate to the sum of column heights $m$ and the variable $t$ conjugate to the length $2N$ of the Dyck path.

To state our results, we define
\begin{equation}
\phi(\lambda,x)=\sum_{k=0}^{\infty}\frac{(i\lambda;q)_k(-i\lambda;q)_k x^k}{(\lambda^2q;q)_k(q;q)_k}=\,_2\phi_1(i\lambda,-i\lambda;\lambda^2 q;q,x)
\end{equation}
and
\begin{equation}
\psi(\lambda,x)=\sum_{k=0}^{\infty}\frac{(i\lambda \sqrt{q};q)_k(-i\lambda\sqrt{q};q)_k x^k}{(\lambda^2q;q)_k(q;q)_k}=\,_2\phi_1(i\lambda\sqrt{q},-i\lambda\sqrt{q};\lambda^2 q;q,x),
\label{psi}
\end{equation}
where
$_2\phi_1(a,b;c;q,x)=\sum_{k=0}^{\infty}\frac{(a;q)_k(b;q)_k\,x^k}{(c;q)_k (q;q)_k}$
is a basic hypergeometric function.
Here, $(a;q)_n=\prod\limits_{k=0}^{n-1}(1-aq^k)$ is the standard notation for the $q$-Pochhammer symbol.

For the path diagrams defined via (\ref{tangentfunc}), we obtain the following theorem.
\begin{theorem}
For $w \geq 0$,
\begin{equation}
G_w(t,q)=\cfrac{1}{1-\dfrac{\lambda^2(1-q)\left[\bar{\lambda}^w\phi(\lambda,q^{3})\phi\left(\bar{\lambda},q^{w+3}\right)-\lambda^w\phi\left(\bar{\lambda},q^{3}\right)\phi(\lambda,q^{w+3})\right]}{{(1+\lambda^2)\left[\bar{\lambda}^w\phi(\lambda,q^{2})\phi\left(\bar{\lambda},q^{w+3}\right)
-\lambda^{w+2}\phi\left(\bar{\lambda},q^{2}\right)\phi(\lambda,q^{w+3})\right]}}}\,,
\label{full solution}
\end{equation}
where $\lambda$ is a root of $\lambda^2-\lambda(1-q)/t+1=0$ and $\bar\lambda=1/\lambda$.
\label{theorem3}
\end{theorem}
Taking the limit $w\to\infty$ we obtain the generating function $G(t,q)$ for unrestricted path diagrams.
\begin{corollary}\label{cort1}
The generating function of $q$-tangent numbers is 
\begin{equation}
G(t,q)=\dfrac{(1+\lambda)^2\left[1-(1+\lambda^2) \sum\limits_{k=0}^{\infty}\dfrac{(-i\lambda)^k}{(1-i\lambda q^{k})} \right]}{\lambda^2 (1-q)}\,,
\label{halfplanefortangent}
\end{equation}
where $\lambda$ is the root of $\lambda^2-\lambda(1-q)/t+1=0$ with smallest modulus.
\end{corollary}
Extracting coefficients of this generating function, we can derive a result equivalent to one obtained previously by different methods \cite[Theorem 1.4]{JosuatVerges20101892}.
\begin{corollary}\label{cor36}
\begin{equation}
[t^{2N}]G(t,q)=\dfrac{1}{(1-q)^{2N+1}}\sum\limits_{m=0}^{N}\dfrac{q^{m^2+2m}\left(
\sum\limits_{l=-m}^{m+1}(-1)^l q^{-l^2+2l}\right)
(2m+2)\binom{2N+1}{N+m+1}}{N+m+2}
\label{tantcoeff}
\end{equation}
\end{corollary}
For the path diagrams defined via (\ref{secantfunc}), we obtain the following theorem.
\begin{theorem}
For $w\geq0$, 
\begin{equation}
G'_w(t,q)=\cfrac{1}{1-\dfrac{\lambda^2(1-q)\left[\bar{\lambda}^w\psi(\lambda,q^{2})\psi\left(\bar{\lambda},q^{w+2}\right)-\lambda^w\psi\left(\bar{\lambda},q^{2}\right)\psi(\lambda,q^{w+2})\right]}{{(1+\lambda^2)\left[\bar{\lambda}^w\psi(\lambda,q)\psi\left(\bar{\lambda},q^{w+2}\right)
-\lambda^{w+2}\psi\left(\bar{\lambda},q\right)\psi(\lambda,q^{w+2})\right]}}}\,,
\label{full solution1}
\end{equation}
where $\lambda$ is the root of $\lambda^2-\lambda(1-q)/t+1=0$ and $\bar{\lambda}=1/\lambda$.
\label{theoremnot3}
\end{theorem}
Taking the limit $w\to\infty$ we obtain the generating function $G'(t,q)$ for unrestricted path diagrams.
\begin{corollary}\label{cors1}
The generating function of $q$-secant numbers is 
\begin{equation}
G'(t,q)=(1+ \lambda^2) \sum\limits_{k=0}^{\infty}\dfrac{(-i\lambda\sqrt{q})^k}{(1-i\lambda \sqrt{q} q^{k})}\;,
\label{halfplaneofsecant}
\end{equation}
where $\lambda$ is the root of $\lambda^2-\lambda(1-q)/t+1=0$ with smallest modulus.
\end{corollary}
Extracting coefficients of this generating function, we can derive a result equivalent to one obtained previously by different methods \cite[Theorem 1.5]{JosuatVerges20101892}.
\begin{corollary}\label{cors2}
\begin{equation}
[t^{2N}]G'(t,q)=\dfrac{1}{(1-q)^{2N}}\sum\limits_{m=0}^{N}\dfrac{q^{m^2+m}\left(\sum\limits_{l=-m}^{m}(-1)^l q^{-l^2}\right)(2m+1)\binom{2N}{N+m}}{N+m+1}.
\label{sectcoeff}
\end{equation}
\end{corollary}

This paper is organized as follows. Section 2 contains further conventions and preliminaries. Section 3 contains the proofs of Theorem \ref{theorem3}
and Corollaries \ref{cort1} and \ref{cor36}. Section 4 contains the proofs of Theorem \ref{theoremnot3} and Corollaries \ref{cors1} and \ref{cors2}. Section 5 contains some novel identities discovered.

\section{Conventions and Preliminaries}
In \cite{Flajolet2006992}, the correspondence between generating functions and continued fractions has been discussed in detail. In particular, in \cite[Theorem 3A]{Flajolet2006992} and  \cite[Theorem 3B]{Flajolet2006992} we find continued fraction expansions for the formal generating functions of path diagrams bounded by a Dyck path with possibility functions given by (\ref{tangentfunc}) and (\ref{secantfunc}), respectively.
It turns out that the counting numbers in these generating functions are the Euler numbers $E_N$, with $G(t,1)$ and $G'(t,1)$ summing over the odd and even Euler numbers, i.e.
\begin{equation}
tG(t,1)=\sum_{N=0}^\infty E_{2N+1}t^{2N+1}=\cfrac{t}{1-\cfrac{1.2 t^2}{1-\cfrac{2.3 t^2}{\ddots}}}\,,
\label{continued fraction}
\end{equation}
and
\begin{equation}
G'(t,1)=\sum_{N=0}^\infty E_{2N}t^{2N}=\cfrac{1}{1-\cfrac{1.1 t^2}{1-\cfrac{2.2 t^2}{\ddots}}}
\label{continued fraction2}
\end{equation}
as formal non-convergent power series.
The odd and even Euler numbers $E_{2N+1}$ and $E_{2N}$ are also known as tangent and secant numbers, respectively, as they occur in the Taylor expansion
\begin{equation}
\tan t+\sec t=\sum_{N=0}^\infty E_N\frac{t^N}{N!}\;.
\end{equation}
The formulas (\ref{continued fraction}) and (\ref{continued fraction2}) were generalised in \cite{JosuatVerges20101892,qtangent} by introducing a variable $q$ conjugate to the sum of the column heights. Briefly, this corresponds to replacing an integer $k$ in (\ref{continued fraction}) and (\ref{continued fraction2}) by the $q$-integer $[k]_q=1+q+\cdots q^k$ (more details are given in Proposition \ref{Prop21}), leading to
\begin{equation}
tG(t,q)=\sum_{N=0}^\infty E_{2N+1}(q)t^{2N+1}=\cfrac{\alpha(1-q)}{1-\cfrac{\alpha^2(1-q)(1-q^2)}{1-\cfrac{\alpha^2(1-q^2)(1-q^3)}{\ddots}}}
\label{Igeneral continued fraction}
\end{equation}
and
\begin{equation}
G'(t,q)=\sum_{N=0}^\infty E_{2N}(q)t^{2N}=\cfrac{1}{1-\cfrac{\alpha^2(1-q)^2}{1-\cfrac{\alpha^2(1-q^2)^2}{\ddots}}}\,,
\label{Igeneral continued fraction2}
\end{equation}
where we have introduced
\begin{equation}
\alpha=\dfrac{t}{1-q}
\label{alpha1}
\end{equation}
and $E_N(q)$ are the $q$-Euler numbers. In particular $E_{2N+1}(q)$ and $E_{2N}(q)$ are known as $q$-tangent and $q$-secant numbers, respectively \cite{qtangent}.
Below we shall use $\alpha$ and $t$ interchangably, as convenient.

The following proposition is the starting point of our analysis. It expresses the height-restricted path diagram generating functions $G_w(t,q)$ and $G'_w(t,q)$ as finite continued fractions.
\begin{proposition}
For $w \geq 0$, 
\begin{equation}
G_w(t,q)=\cfrac{1}{1-\cfrac{\alpha^2(1-q)(1-q^2)}{1-\cfrac{\alpha^2(1-q^2)(1-q^3)}{{\ddots-}{\cfrac{\alpha^2(1-q^{w-1})(1-q^{w})}{1-\alpha^2(1-q^w)(1-q^{w+1})}}}}}
\label{general continued fraction}
\end{equation}
and
\begin{equation}
G'_w(t,q)=\cfrac{1}{1-\cfrac{\alpha^2(1-q)^2}{1-\cfrac{\alpha^2(1-q^2)^2}{\ddots-\cfrac{\alpha^2(1-q^{w-1})^2}{1-\alpha^2(1-q^w)^2}}}}\, .
\label{general continued fraction2}
\end{equation}
\label{Prop21}
\end{proposition}              
\begin{proof}

From the combinatorial theory of continued fractions given in \cite{Flajolet2006992}, if $X=(a_0,a_1,a_2,..,b_0,b_1,..)$ then the Stieltjes type continued fraction is  
$$S_k(X,t)=\cfrac{1}{1-\cfrac{a_0 b_1 t^2}{1-\cfrac{a_1 b_2 t^2}{\ddots-\cfrac{a_{k-2}b_{k-1}t^2}{1-a_{k-1} b_k t^2}}}}$$
where $a_i$ corresponds to the weight of a northeast step starting at height $i$, $b_j$ corresponds to the weight of a southeast step starting at height $j$, and $t$ is conjugate to the length of the Dyck path. 
Hence, we only need to specify the weights $a_i$ and $b_j$.

Possible column heights below a northeast step starting at height $i$ range from $0$ to $i$, and hence $a_i=1+q+\ldots+q^i$. For $G_w$ possible column heights below a southeast step starting at height $j$ range from $0$ to $j$, and hence $b_j=1+q+\ldots+q^j$, whereas for $G'_w$ possible column heights below a southeast step starting at height $j$ range from $0$ to $j-1$, and hence $b_j=1+q+\ldots+q^{j-1}$.
\end{proof}

It is obvious that we can write the right-hand sides of \eqref{general continued fraction} and \eqref{general continued fraction2} as rational functions. 
\begin{proposition}
For $w\geq 0$, 
\begin{equation}
G_w(t,q)=\dfrac{P_w(\alpha,q)}{Q_w(\alpha,q)}\quad\text{and}\quad G'_w(t,q)=\dfrac{P'_w(\alpha,q)}{Q'_w(\alpha,q)} \, ,
\label{rational function}
\end{equation}
where
\begin{eqnarray}
  P_w&=&
  \begin{cases}
    0 & w=-1 \\
    1 & w=0\\
    P_{w-1}-\alpha^2(1-q^w)(1-q^{w+1})P_{w-2}& w\geq 1
  \end{cases}\;,
  \label{recursion P}
\\
  Q_w&=&
  \begin{cases}
    1 & w=-1 \\
    1 & w=0\\
    Q_{w-1}-\alpha^2(1-q^w)(1-q^{w+1})Q_{w-2}& w\geq 1
  \end{cases}\;,
  \label{recursion Q}
\\
  P'_w&=&
  \begin{cases}
    0 & w=-1 \\
    1 & w=0\\
    P_{w-1}-\alpha^2(1-q^w)^2 P_{w-2}& w\geq 1
  \end{cases}\qquad\text{and}
  \label{recursion P1}
\\
  Q'_w&=&
  \begin{cases}
    1 & w=-1 \\
    1 & w=0\\
    Q_{w-1}-\alpha^2(1-q^w)^2 Q_{w-2}& w\geq 1
  \end{cases}\;.
  \label{recursion Q1}
\end{eqnarray}
\label{proposition2}
\end{proposition}
\begin{proof}
The initial conditions follow from the fact that $G_{-1}(t,q)=G'_{-1}(t,q)=0/1$. This implies that $P_{-1}=P'_{-1}=0$ and $Q_{-1}=Q'_{-1}=1$. Also for $w=0$ we have $G_{0}(t,q)=G'_{0}(t,q)=1/1$. For $w \geq 1$ we compare with the $h$-th convergent of the $J$-fraction on page 152 of \cite{Flajolet2006992}. We have $z=t$ and $a_k=1$ for $k \geq 1 $ and $b_k=(1-q^w)(1-q^{w+1})$ and $c_k=0$ for $k \geq0$. This reduces to the recurrence equations given in \eqref{recursion P} and \eqref{recursion Q}. For the generating function $G'_w(t,q)$ we see that, instead, $b_k=(1-q^w)(1-q^{w})$, which results in the recurrence equations given in \eqref{recursion P1} and \eqref{recursion Q1}.
\end{proof} 
\section{$q$-tangent numbers}
We shall prove Theorem \ref{theorem3} by solving the recurrence relations \eqref{recursion P} and \eqref{recursion Q}.
We can write $P_w$ and $Q_w$ as the linear combination of two basic hypergeometric functions and determine the coefficients from the initial conditions  of the recurrences given in Proposition \ref{proposition2}.

\begin{proof}[Proof of Theorem \ref{theorem3}]
For $w \geq 1$ the recurrence relations for $P_w(\alpha,q)$ and $Q_w(\alpha,q)$ are the same, so we represent them both by $R(w)$ and solve simultaneously. 
From the recursion given in \eqref{recursion P} and \eqref{recursion Q} we have for $w\geq1$,
\begin{equation}
R(w)=R(w-1)-\alpha^2(1-q^w)(1-q^{w+1})R(w-2).
\label{recurrence}
\end{equation}
\indent 
Unlike a linear recurrence with constant coefficients, this cannot be solved by a standard method because we have $w$-dependent coefficients. Moreover, the occurrence of both $q^w$ and $q^{2w}$ poses a difficulty, so our next step will be to eliminate the term containing $q^{2w}$ by appropriate rewriting of the recurrences. It is evident from the coefficient of $R(w-2)$ that multiplying by a $q$-factorial will simplify \eqref{recurrence} appropriately. Rescaling the recursion \eqref{recurrence} by substituting
\begin{equation}
R(w)=\alpha^w(q;q)_{w+1}S(w)
\label{ansatz1}
\end{equation}
leads to the recurrence
\begin{equation}
S(w)-\frac1\alpha S(w-1)+S(w-2)=q^{w+1}(S(w)+S(w-2))
\label{scaled recurrence}
\end{equation}
for $w \geq 1$.
This eliminates $q^{2w}$ from the recurrence as intended, as the right hand side only contains a $q^w$ prefactor.
The left hand side of equation \eqref{scaled recurrence} is a linear homogeneous recurrence relation with a characteristic polynomial
\begin{equation}
P(\lambda)=\lambda^2-\frac{\lambda}{\alpha}+1.
\label{polynomial}
\end{equation}
The two roots $\lambda_1$ and $\lambda_2$ of the characteristic polynomial are reciprocal to each other,
\begin{equation}
\label{prod of roots}
\lambda_1 \lambda_2=1,
\end{equation}
a fact that we will need to use below.
If the right hand side of the recurrence relation \eqref{scaled recurrence} was zero then the solution could be written as a $q$-independent linear combination of the powers of the roots of the characteristic polynomial. To solve the recurrence \eqref{scaled recurrence} in general, we use the ansatz
\begin{equation}
 S(w)=\lambda^w\sum\limits_{k=0}^{\infty}c_k q^{kw},
 \label{ansatz2}
\end{equation} 
which has been shown to work when there are powers of $q^w$ in such a linear recurrence \cite{RSOS,RSOSslit}.
The recurrence relation for $c_k$ can then be read off from
\begin{equation}
P(\lambda)c_0+\sum_{k=1}^{\infty}q^{kw-2k}\left(P(\lambda q^k)c_k-(\lambda^2 q^{2k}+ q^2)qc_{k-1}\right)=0.
\label{Two term recurrence}
\end{equation}
This equation is satisfied if $P(\lambda)=0$ and all the coefficients in the sum vanish, i.e.  $P(\lambda q^k)c_k-(\lambda^2 q^{2k}+q^2)qc_{k-1}=0$. 
The latter condition implies 
\begin{equation}
c_k=\frac{(\lambda^2 q^{2k}+q^2)qc_{k-1}}{P(\lambda q^k)}\, .
\label{ck}
\end{equation}
The condition $P(\lambda)=0$ enables us to express $\alpha$ in terms of $\lambda$ as
$\alpha={\lambda}/(1+\lambda^2)$,
and eliminating $\alpha$ in the characteristic polynomial \eqref{polynomial}, we find
\begin{equation}
P(\lambda q^k)=(1-q^k)(1-\lambda^2 q^k).
\label{plambdaq}
\end{equation}
Now substituting in the value of $P(\lambda q^k)$ from  \eqref{plambdaq} in \eqref{ck} and iterating it we have 
\begin{equation}
c_k=\frac{(-\lambda^2;q^2)_k\, q^{3k}}{(q;q)_k(\lambda^2 q;q)_k}
,
\label{Hypergeometric Ck}
\end{equation}
where we choose to write all products in terms of the $q$-Pochhammer symbol.
The full solution to the recurrence equation \eqref{scaled recurrence} is a linear combination of the ansatz (\ref{ansatz2}) over both roots of $P(\lambda)$. As $P(\lambda)=0$ implies $P(\bar{\lambda})=0$,
we can write the general solution for $S(w)$ as
\begin{equation}
S(w)=A\lambda^w\sum_{k=0}^{\infty}c_k(\lambda,q) q^{kw}+B\bar{\lambda}^w\sum_{k=0}^{\infty}c_k\left(\bar{\lambda},q\right)q^{kw}.
\end{equation}
We can now write the general solution in terms of a basic hypergeometric series by defining
\begin{equation}
\phi(\lambda,x)=\sum_{k=0}^{\infty}\frac{(-\lambda^2;q^2)_k x^{k}}{(q;q)_k(\lambda^2 q;q)_k}
=\,_2\phi_1(i\lambda,-i\lambda;\lambda^2 q;q,x),
\label{phi}
\end{equation}
where
$$_2\phi_1(a,b;c;q,x)=\sum_{k=0}^{\infty}\frac{(a;q)_k(b;q)_k\,x^k}{(c;q)_k (q;q)_k}.$$
Using this notation, the general solution $S(w)$ can simply be written as
\begin{equation}
S(w)=A\lambda^w\phi(\lambda,q^{w+3})+B\bar{\lambda}^w\phi\left(\bar{\lambda},q^{w+3}\right).
\label{linear general solution}
\end{equation}
Using the initial conditions 
$$S(-1)=0\qquad	S(0)=\frac{1}{1-q}$$
derived from \eqref{recursion P} and solving for $A$ and $B$, 
we get 
$$A=\frac{-\lambda^2\phi\left(\bar{\lambda},q^2\right)}{(1-q)\left(\phi(\lambda,q^2)\phi\left(\bar{\lambda},q^3\right)-\lambda^2\phi\left(\bar{\lambda},q^2\right)\phi(\lambda,q^3)\right)}$$
and
$$B=\frac{\phi(\lambda,q^2)}{(1-q)\left(\phi(\lambda,q^2)\phi\left(\bar{\lambda},q^3\right)-\lambda^2\phi\left(\bar{\lambda},q^2\right)\phi(\lambda,q^3)\right)}\, .$$
Similarly, using the initial conditions 
$$
S(-1)=\alpha\qquad	S(0)=\frac{1}{1-q}
$$
derived from \eqref{recursion Q} we get 
$$A=\frac{(\alpha)(\lambda)(1-q)\phi\left(\bar{\lambda},q^3\right)-\lambda^2\phi\left(\bar{\lambda},q^2\right)}{(1-q)\left(\phi(\lambda,q^2)\phi\left(\bar{\lambda},q^3\right)-\lambda^2\phi\left(\bar{\lambda},q^2\right)\phi(\lambda,q^3)\right)}$$
and
$$B=\frac{\phi(\lambda,q^2)-(\alpha)(\lambda)(1-q)\phi(\lambda,q^3)}{(1-q)\left(\phi(\lambda,q^2)\phi\left(\bar{\lambda},q^3\right)
-\lambda^2\phi\left(\bar{\lambda},q^2\right)\phi(\lambda,q^3)\right)}\,.$$
Substituting the full solution for $P_w(\alpha,q)$ and $Q_w(\alpha,q)$ in \eqref{rational function}, we arrive at the expression given in \eqref{full solution}. This completes the proof.
\end{proof}
By taking the limit of infinite $w$ in the generating function $G_w$, we derive an expression for the generating function of $q$-tangent numbers.
\begin{proof}[Proof of Corollary \ref{halfplanefortangent}]
We consider the right-hand side of \eqref{full solution}. We know that the basic hypergeometric functions converge when $|q| < 1$ using the ratio test. From \eqref{prod of roots} we see that one of the roots of the characteristic polynomial \eqref{polynomial} is less than one if $t$ is sufficiently small. We therefore choose the root $\lambda$ such that $|\lambda| < 1$.
When $w \rightarrow \infty$,
$$\phi(\lambda,q^{w+3})=\,_2\phi_1(i\lambda,-i\lambda;\lambda^2q;q,q^{w+3})\rightarrow \,_2\phi_1(i\lambda,-i\lambda;\lambda^2q;q,0)=1.$$ Also $$ |\lambda^w| \rightarrow 0.$$
This implies
\begin{equation}
G(t,q)=\cfrac{1}{1-\dfrac{\lambda^2(1-q)\phi(\lambda,q^{3})}{(1+\lambda^2)\phi(\lambda,q^{2})}}\, .
\label{HP}
\end{equation}
Heine's transformation formula for $_2\phi_1$ series \cite{Hein} is given by
\begin{equation}
_2\phi_1(a,b;c;q,z)=\frac{(b;q)_\infty(az;q)_\infty}{(c;q)_\infty (z;q)_\infty}\,_2\phi_1(c/b,z;az;q,b).
\label{Hein}
\end{equation}
Using this transformation we can write the basic hypergeometric functions in \eqref{HP} as follows
\begin{equation}
\phi(\lambda,q^2)
=\frac{(-i\lambda;q)_\infty(i\lambda q^2;q)_\infty}{(\lambda^2q;q)_\infty (q^2;q)_\infty} \,_2\phi_1(i\lambda q,q^2;i\lambda q^2;q,-i\lambda)
\label{fihein1}
\end{equation}
and 
\begin{equation}
\phi(\lambda,q^3)
=\frac{(-i\lambda;q)_\infty(i\lambda q^3;q)_\infty}{(\lambda^2q;q)_\infty (q^3;q)_\infty} \,_2\phi_1(i\lambda q,q^3;i\lambda q^3;q,-i\lambda).
\label{fihein2}
\end{equation}
Further substituting the transformations of basic hypergeometric functions from \eqref{fihein1} and \eqref{fihein2}  into \eqref{HP} yields
\begin{equation}
G(t,q)=\cfrac{1}{1-\dfrac{\lambda^2(1-q)\dfrac{(-i\lambda;q)_\infty(i\lambda q^3;q)_\infty}{(\lambda^2q;q)_\infty (q^3;q)_\infty} \,_2\phi_1(i\lambda q,q^3;i\lambda q^3;q,-i\lambda)}{(1+\lambda^2)\dfrac{(-i\lambda;q)_\infty(i\lambda q^2;q)_\infty}{(\lambda^2q;q)_\infty (q^2;q)_\infty} \,_2\phi_1(i\lambda q,q^2;i\lambda q^2;q,-i\lambda)}}\, .
\end{equation}
Expressing these basic hypergeometric functions by their explicit sums, we find that many factors in the coefficients cancel:
\begin{align}
G(t,q)&=\cfrac{1}{1-\dfrac{\lambda^2(1-q)(1-q^2)\sum\limits_{k=0}^{\infty}\dfrac{(i\lambda q;q)_k(q^3;q)_k}{(i\lambda q^3;q)_k(q;q)_k}(-i\lambda)^k}{(1-i\lambda q^2)(1+\lambda^2)\sum\limits_{k=0}^{\infty}\dfrac{(i\lambda q;q)_k(q^2;q)_k}{(i\lambda q^2;q)_k(q;q)_k}(-i\lambda)^k }}\newline\\
&=\cfrac{1}{1-\dfrac{(\lambda^2)(1-q)\sum\limits_{k=0}^{\infty}\dfrac{(1-q^{k+1})(1-q^{k+2})}{(1-i\lambda q^{k+1})(1-i\lambda q^{k+2})}(-i\lambda)^k}{(1+\lambda^2)\sum\limits_{k=0}^{\infty}\dfrac{(1-q^{k+1})}{(1-i\lambda q^{k+1})}(-i\lambda)^k}}\, .
\label{simple1}
\end{align}
We next aim to simplify the terms in the sums on the right hand side of \eqref{simple1}. For this we let
\begin{equation}
N=\frac{(1-q^{k+1})(1-q^{k+2})}{(1-i\lambda q^{k+1})(1-i\lambda q^{k+2})}(-i\lambda)^k
\label{N}
\end{equation}
and 
\begin{equation}
D=\frac{(1-q^{k+1})}{(1-i\lambda q^{k+1})}(-i\lambda)^k.
\label{D}
\end{equation}
We substitute $x=-i\lambda$ and employ partial fraction expansion with respect to $q^k$. Shifting summation indices and combining fractions, we find
\begin{multline}
N=(-1)\dfrac{x^4\sum\limits_{k=0}^{\infty}\dfrac{x^k}{(1+x q^{k+1})}-2 x^2\sum\limits_{k=0}^{\infty}\dfrac{x^k}{(1+x q^{k+1})}}{x^4 (q-1)(x-1)}\\
-\dfrac{\sum\limits_{k=0}^{\infty}\dfrac{x^k}{(1+xq^{k+1})}+x^2 q+1}{x^4 (q-1)(x-1)}
\label{p1}
\end{multline}
and
\begin{equation}
D=\dfrac{x^2\sum\limits_{k=0}^{\infty}\dfrac{x^k}{(1+x q^{k})}-\sum\limits_{k=0}^{\infty}\dfrac{x^k}{(1+x q^{k})}+1}{x^2 (x-1)}.
\label{p2}
\end{equation}
Substituting \eqref{p1} and \eqref{p2} into \eqref{simple1} and simplifying, we get the final expression \eqref{halfplanefortangent}.
\end{proof}
Next, we extract the coefficient of $t^{2N}$ of $G(t,q)$ given in \eqref{halfplanefortangent}. To start, we need an identity which can be obtained from counting rectangles on the square lattice in two different ways, taking ideas from \cite{thomas}.
\begin{lemma}
\label{rectlemma}
\begin{equation}
\label{rectangles}
\sum_{n=0}^\infty\frac{x^n}{1-yq^n}=\sum_{n=0}^\infty\frac{x^ny^nq^{n^2}(1-xyq^{2n})}{(1-xq^n)(1-yq^n)}\;.
\end{equation}
\end{lemma}
\begin{proof}
We consider the generating function of rectangles (including those of height or width zero) on the square lattice, counted with respect to height, width, and area, given by
\begin{equation}
R(x,y,q)=\sum_{n,m=0}^\infty x^ny^mq^{nm}\;.
\end{equation}
Summing over $m$ gives the left hand side of identity \eqref{rectangles}. If we instead sum over rectangles of fixed minimal width or height $N$, then this gives the right hand side of identity \eqref{rectangles}.
\end{proof}
\begin{proof}[Proof of Corollary \ref{cor36}]
The sum in \eqref{halfplanefortangent} can be identified with $R(-i\lambda,i\lambda,q)$, so that using Lemma \ref{rectlemma} we get
\begin{align}
G(t,q)=&\dfrac{(1+\lambda)^2\left[1-(1+\lambda^2) R(-i\lambda,i\lambda,q) \right]}{\lambda^2 (1-q)}\nonumber\\
=&\dfrac{(1+\lambda)^2\left[1-(1+\lambda^2) \sum\limits_{n=0}^{\infty}\dfrac{q^{n^2} \lambda^{2n} (1-\lambda^2 q^{2n}) }{(1+ \lambda^2 q^{2n})} \right]}{\lambda^2 (1-q)}\;.
\label{tangentt}
\end{align}
We remind that $G(t,q)$ is by definition an even function in $t$, and that the $t$-dependence on the right hand side is implicit in $\lambda=\lambda(t)$ by \eqref{polynomial} and \eqref{alpha1}. To extract the coefficient of $t^{2N}$, we evaluate the contour integral
\begin{equation}
[t^{2N}] G(t,q)=\frac{1}{2 \pi i}
\oint \dfrac{G(t,q)}{t^{2N+1}}dt\;.
\label{cit}
\end{equation}
Using the variable substitution
\begin{equation}
t=\frac{(1-q)\lambda}{1+\lambda^2},
\label{valuet}
\end{equation}
we get
\begin{multline}
[t^{2N}] G(t,q)=\\ \frac{1}{2 \pi i}\oint \left(\dfrac{(1+\lambda^2)^{2N} \left(1-(1+\lambda^2)\sum\limits_{n=0}^{\infty}\dfrac{q^{n^2} \lambda^{2n} (1-\lambda^2 q^{2n}) }{(1+\lambda^2 q^{2n})}\right)(1-\lambda^2)}{\lambda^{2N+2}(1-q)^{2N+1}}\right)\frac{d\lambda}{\lambda}.
\end{multline}
We thus have
\begin{equation}
[t^{2N}]G\left(t,q\right)=[\lambda^0]H_N(\lambda,q),
\end{equation}
and to extract the constant term in $\lambda$ on the right-hand side we now expand $H_N(\lambda,q)$ as a Laurent series in $\lambda$. For this, we write
\begin{equation}
[\lambda^0]H_N(\lambda,q)=[\lambda^0]\frac{T_1-T_2\sum\limits_{n=0}^{\infty}T_3}{((1-q)^{2N+1}}\;,
\label{constantterm0}
\end{equation}
where
$$
T_1=\left(\lambda+\dfrac{1}{\lambda}\right)^{2N} \left(\dfrac{1}{\lambda^2}-1\right),
$$
$$T_2=\left(\lambda+\dfrac{1}{\lambda}\right)^{2N} \left(\dfrac{1}{\lambda^2}-1\right)(1+\lambda^2),
$$
and
$$T_3=\dfrac{q^{n^2} \lambda^{2n} (1-\lambda^2 q^{2n}) }{(1+\lambda^2 q^{2n})}.$$
We find the series expansions
\begin{equation}
T_1=\sum\limits_{k=0}^{2N+1}\dfrac{(2N)!(2N-2k+1)\lambda^{2k-2N-2}}{k!(2N-k+1)!}\, ,
\label{term0}
\end{equation}
\begin{equation}
T_2=\sum\limits_{k=0}^{2N+2}\dfrac{(2N+1)!(2N-2k+2)\lambda^{2k-2N-2}}{k!(2N-k+2)!}
\label{term1}
\end{equation}
and
\begin{equation}
T_3=q^{n^2} \lambda^{2n}\left(2\left(\sum\limits_{l=0}^{\infty}(-1)^l(\lambda^2 q^{2n})^l\right)-1\right).
\label{term2}
\end{equation}
Next we substitute the expression \eqref{term0}, \eqref{term1} and \eqref{term2} in \eqref{constantterm0}, which after some simplification leads to
\begin{multline}
H_N(\lambda,q)=\dfrac{1}{(1-q)^{2N+1}}\left(\sum\limits_{k=0}^{2N+1}\dfrac{(2N)!(2N-2k+1)\lambda^{2k-2N-2}}{k!(2N-k+1)!}\right.\\
\left.-2\sum_{n=0}^{\infty}\sum_{l=0}^{\infty}\sum_{k=0}^{2N+2}(-1)^l q^{n^2+2nl} \,\dfrac{(2N+1)!(2N-2k+2)\lambda^{2k-2N-2+2n+2l}}{k!(2N-k+2)!}\right.\\
\left.+\sum_{n=0}^{\infty}\sum_{k=0}^{2N+2}q^{n^2}\dfrac{(2N+1)!(2N-2k+2)\lambda^{2k-2N-2+2n}}{k!(2N-k+2+2n)!}\right)\;.
\label{2.65}
\end{multline}
We want to extract the constant term in $\lambda$, so we combine the powers of $\lambda$ and equate them to $0$. This fixes the summation index $k$, and we get 
\begin{multline}
[t^{2N}]G\left(t,q\right)
=\dfrac{1}{(1-q)^{2N+1}}\left(-\dfrac{(2N)!}{N!(N+1)!}\right.\\
\left.
-2\sum_{n=0}^{N+1}\sum_{l=0}^{N-n+1}(-1)^l q^{n^2+2nl}\dfrac{(2N+1)!(2n+2l)}{(N-n-l+1)!(N+n+l+1)!}\right.\\ \left.+\sum_{n=0}^{N+1} q^{n^2}\dfrac{(2N+1)!(2n)}{(N-n+1)!(N+n+1)!}\right).
\end{multline}
Completing the square in the middle sum and changing summation indices leads to
\begin{multline}
[t^{2N}]G\left(t,q\right)=\dfrac{1}{(1-q)^{2N+1}}\left(-\dfrac{(2N)!}{N!(N+1)!}\right.\\
\left.-\sum_{m=0}^{N+1}q^{m^2}\dfrac{(2N+1)!(2m)}{(N-m+1)!(N+m+1)!}\left(\sum_{l=-(m+1)}^{m+1}(-1)^l q^{-l^2}\right)\right)\;,
\end{multline}
where in a final step we combined the last two sums. Shifting summation indices $m$ and $l$ gives
\begin{multline}
[t^{2N}]G\left(t,q\right)=\dfrac{1}{(1-q)^{2N+1}}\left(-\dfrac{(2N)!}{N!(N+1)!}\right.\\
\left.+\sum_{m=0}^{N}q^{m^2+2m}\dfrac{(2N+1)!(2m+2)}{(N-m)!(N+m+2)!}\left(\sum_{l=-m}^{m+2}(-1)^{l} q^{-l^2+2l}\right)\right).
\end{multline}
Performing the sum over $m$ with $l=m+2$
cancels the first term and we arrive at an expression equivalent to \eqref{tantcoeff},
\begin{multline}
[t^{2N}]G\left(t,q\right)=\dfrac{1}{(1-q)^{2N+1}}\times\\ \sum_{m=0}^{N}q^{m^2+2m}\dfrac{(2N+1)!(2m+2)}{(N-m)!(N+m+2)!}\left(\sum_{l=-m}^{m+1}(-1)^{l} q^{-l^2+2l}\right).
\end{multline}
\end{proof}
\section{$q$-secant numbers}
We begin by giving the proof of Theorem \ref{theoremnot3}. We shall prove it by solving the recurrences \eqref{recursion P1} and \eqref{recursion Q1}. This is done along the same lines as in the proof of Theorem \ref{theorem3}. 
\begin{proof}
It follows from the continued fraction expansion given in \eqref{general continued fraction2} that both the numerator $P'_w(\alpha,q)$ and denominator $Q'_w(\alpha,q)$ satisfy the recurrence relations given in \eqref{recursion P1} and \eqref{recursion Q1} respectively. As the recursions are the same for $w \geq 1$, we represent them both by $R(w)$ and solve simultaneously. It follows that 
\begin{equation}
R(w)=R(w-1)-\alpha^2(1-q^w)^2 R(w-2).
\end{equation}
Expanding the coefficient of $R(w-2)$ gives three terms which cannot be solved explicitly using standard methods because we have $w$-dependent coefficients. Also the terms $q^w$ and $q^{2w}$ cause difficulty, so we will aim to eliminate the terms containing $q^{2w}$  by suitable rescaling. For this we use the ansatz \eqref{ansatz1}. This transformation of coefficients leads to
\begin{equation}
S(w)-\frac{1}{\alpha}S(w-1)+S(w-2)=q^{w+1}S(w)+q^w S(w-2)
\label{scaled recurrence1}
\end{equation}
for $w \geq 1$. This eliminates the $q^{2w}$ from the recurrence as intended, with only $q^w$ factors on the right hand side. We see that this recurrence is very similar to \eqref{scaled recurrence}.
The left hand side of \eqref{scaled recurrence1} is a linear homogeneous recurrence relation with the same characteristic polynomial \eqref{polynomial} as above, however the right hand side is slightly different, with a prefactor of $q^w$ in front of $S(w-2)$ instead of a prefactor $q^{w+1}$. We thus use the same ansatz \eqref{ansatz2} to solve the recurrence. Following a calculation identical to the one for $q$-tangent numbers, we find for $k>0$
\begin{equation}
c_k=\frac{(\lambda^2 q^{2k}+q)qc_{k-1}}{P(\lambda q^k)}.
\label{ck1}
\end{equation}
Now substituting the value of $P(\lambda q^k)$ in \eqref{ck1} and iterating it, we get
\begin{equation}
c_k=\frac{(-\lambda^2q;q^2)_k\, q^{2k}}{(q;q)_k(\lambda^2 q;q)_k}\;.
\label{Hypergeometric ck1}
\end{equation}
The full solution to the recurrence equation \eqref{scaled recurrence1} is a linear combination of the ansatz over both the values of $\lambda$. Here $P(\lambda)=0$ and also $P(\bar{\lambda})=0$ (where $\bar{\lambda}=\frac{1}{\lambda}$). We can write the general solution for $S(w)$ as
\begin{equation}
S(w)=A\lambda^w\sum_{k=0}^{\infty}c_k(\lambda,q) q^{kw}+B\bar{\lambda}^w\sum_{k=0}^{\infty}c_k\left(\bar{\lambda},q\right)q^{kw}.
\end{equation}
We define 
\begin{align*}
\psi(\lambda,x)=&
\sum_{k=0}^{\infty}\frac{(-\lambda^2q;q^2)_k\, x^{k}}{(q;q)_k(\lambda^2 q;q)_k}
=\sum_{k=0}^{\infty}\frac{(i\lambda \sqrt{q};q)_k(-i\lambda \sqrt{q};q)_k x^k}{(\lambda^2q;q)_k(q;q)_k}\\
=&
\,_2 \phi_1(i\lambda \sqrt{q},-i\lambda \sqrt{q};\lambda^2 q;q,x)
\end{align*}
where $_2 \phi_1$ is a basic hypergeometric function. The general solution can be expressed as follows
\begin{equation}
S(w)=A\lambda^n\psi(\lambda,q^{w+2})+B\bar{\lambda}^w\psi\left(\bar{\lambda},q^{w+2}\right).
\label{linear general solution1}
\end{equation}
Using the initial conditions, we can solve for $A$ and $B$. First we solve it for $P'_w$ with the initial conditions as
$$S(-1)=0\qquad	S(0)=\frac{1}{1-q}\;.$$
Substituting these initial conditions into equation \eqref{linear general solution1} and solving it for $A$ and $B$, we have 
$$A=\frac{-\lambda^2\psi\left(\bar{\lambda},q\right)}{(1-q)\left(\psi(\lambda,q)\psi\left(\bar{\lambda},q^2\right)-\lambda^2\psi\left(\bar{\lambda},q\right)\psi(\lambda,q^2)\right)}$$
and
$$B=\frac{\psi(\lambda,q)}{(1-q)\left(\psi(\lambda,q)\psi\left(\bar{\lambda},q^2\right)-\lambda^2\psi\left(\bar{\lambda},q\right)\psi(\lambda,q^2)\right)}.$$
Similarly, we solve for $Q'_w$ using the initial conditions
$$S(-1)=\alpha\qquad\text{and}\qquad	S(0)=\frac{1}{1-q}.$$
We obtain
$$A=\frac{(\alpha)(\lambda)(1-q)\psi\left(\bar{\lambda},q^2\right)-\lambda^2\psi\left(\bar{\lambda},q\right)}{(1-q)\left(\psi(\lambda,q)\psi\left(\bar{\lambda},q^2\right)-\lambda^2\psi\left(\bar{\lambda},q\right)\psi(\lambda,q^2)\right)}$$
and
$$B=\frac{\psi(\lambda,q)-(\alpha)(\lambda)(1-q)\psi(\lambda,q^2)}{(1-q)\left(\psi(\lambda,q)\psi\left(\bar{\lambda},q^2\right)-\lambda^2\psi\left(\bar{\lambda},q\right)\psi(\lambda,q^2)\right)}.$$
Substituting the full solution for $P'_w(\alpha,q)$ and $Q'_w(\alpha,q)$ in \eqref{rational function} we have the expression in \eqref{full solution1}.
\end{proof}
By taking the limit of infinite $w$ in the generating function $G'_w$, we derive an expression for the generating function of $q$-secant numbers. 
\begin{proof}[Proof of Corollary \ref{cors1}]
Consider the right hand side of \eqref{full solution1}. As above, we choose $\lambda$ to be the smaller root of the characteristic polynomial \eqref{polynomial}. For $w \rightarrow \infty$ we have 
\begin{multline}
\psi(\lambda,q^{w+2})=\,_2\phi_1(i\lambda\sqrt{q},-i\lambda\sqrt{q};\lambda^2q;q,q^{w+2})\\ \rightarrow \,_2\phi_1(i\lambda\sqrt{q},-i\lambda\sqrt{q};\lambda^2q;q,0)=1\end{multline}
and 
$$ |\lambda^w| \rightarrow 0\;.$$
This implies
\begin{equation}
G'(t,q)=\cfrac{1}{1-\dfrac{\lambda^2(1-q)\psi(\lambda,q^{2})}{(1+\lambda^2)\psi(\lambda,q)}}.
\label{HP1}
\end{equation}
Using Heine's transformation formula given in \eqref{Hein}, we transform the basic hypergeometric functions given in \eqref{HP1} as follows
\begin{align}
\psi(\lambda,q)
&=\frac{(-i\lambda\sqrt{q};q)_\infty(i\lambda q^{3/2};q)_\infty}{(\lambda^2q;q)_\infty(q;q)_\infty} \,_2\phi_1(i\lambda \sqrt{q},q;i\lambda q^{3/2};q,-i\lambda\sqrt{q})
\label{psiq}
\end{align}
and
\begin{align}
\psi(\lambda,q^2)
&=\frac{(-i\lambda\sqrt{q};q)_\infty(i\lambda  q^{5/2};q)_\infty}{(\lambda^2q;q)_\infty(q^2;q)_\infty} \,_2\phi_1(i\lambda \sqrt{q},q^2;i\lambda q^{5/2};q,-i\lambda\sqrt{q}).
\label{psiq2}
\end{align}
Substituting the transformations \eqref{psiq} and \eqref{psiq2} into the limit \eqref{HP1} yields
\begin{multline}
G'(t,q)=\\\cfrac{1}{1-\dfrac{\lambda^2(1-q)\dfrac{(-i\lambda\sqrt{q};q)_\infty(i\lambda q^{5/2};q)_\infty}{(\lambda^2q;q)_\infty (q^2;q)_\infty} \,_2\phi_1(i\lambda\sqrt{q},q^2;i\lambda q^{5/2};q,-i\lambda\sqrt{q})}{(1+\lambda^2)\dfrac{(-i\lambda\sqrt{q};q)_\infty(i\lambda q^{3/2};q)_\infty}{(\lambda^2q;q)_\infty(q;q)_\infty} \,_2\phi_1(i\lambda \sqrt{q},q;i\lambda q^{3/2};q,-i\lambda\sqrt{q})}}\,.
\end{multline}
Expressing these basic hypergeometric functions by their explicit sums, we find
\begin{align}
G'(t,q)&=\cfrac{1}{1-\dfrac{\lambda^2(1-q)^2\sum\limits_{k=0}^{\infty}\dfrac{(i\lambda\sqrt{q};q)_k(q^2;q)_k}{(i\lambda q^{5/2};q)_k(q;q)_k}(-i\lambda\sqrt{q})^k}{(1+\lambda^2)(1-i\lambda
q^{3/2})\sum\limits_{k=0}^{\infty}\dfrac{(i\lambda \sqrt{q};q)_k(q;q)_k}{(i\lambda q^{3/2};q)_k(q;q)_k}(-i\lambda\sqrt{q})^k }}\newline \\
&=\cfrac{1}{1-\dfrac{\lambda^2(1-q)\sum\limits_{k=0}^{\infty}\dfrac{(1-q^{k+1})}{(1-i\lambda q^{k+1/2})(1-i\lambda q^{k+3/2})}(-i\lambda\sqrt{q})^k}{(1+\lambda^2)\sum\limits_{k=0}^{\infty}\dfrac{(-i\lambda\sqrt{q})^k}{(1-i\lambda q^{k+1/2})} }}\;.
\label{simple2}
\end{align}
To simplify further we consider the expression in \eqref{simple2}. We aim to simplify the terms in the sums on the right hand side of \eqref{simple2}. For this we let
\begin{equation}
N=\sum\limits_{k=0}^{\infty}\dfrac{(1-q^{k+1})}{(1-i\lambda q^{k+1/2})(1-i\lambda q^{k+3/2})}(-i\lambda\sqrt{q})^k
\end{equation}
and 
\begin{equation}
D=\dfrac{(-i\lambda\sqrt{q})^k}{(1-i\lambda q^{k+1/2})} \;.
\end{equation}
We substitute $x=-i\lambda\sqrt{q}$ and apply partial fraction decomposition to get
\begin{equation}
N=\dfrac{(q-x^2)\sum\limits_{k=0}^{\infty}\dfrac{x^k}{1+xq^k}}{x^2(q-1)}-\frac{q}{x^2(q-1)}\;
\label{N2}
\end{equation}
and
\begin{equation}
D=\sum\limits_{k=0}^{\infty}\frac{x^k}{1+xq^k}\;.
\label{d2}
\end{equation}
Substituting \eqref{N2} and \eqref{d2} into \eqref{simple2} and simplifying, we get the final result as \eqref{halfplaneofsecant}.
\end{proof}
Next we extract the coefficient of $t^{2N}$ of $G'(t,q)$ given in \eqref{halfplaneofsecant}. 
\begin{proof}[Proof of corollary \ref{cors2}]
To prove this corollary we will again use the Lemma \ref{rectlemma}. The sum in \eqref{halfplaneofsecant} can be identified with $R(-i\lambda \sqrt{q},i\lambda \sqrt{q},q)$, so we get
\begin{align}
G'(t,q)=&(1+\lambda)^2R(-i\lambda\sqrt{q},i\lambda\sqrt{q},q)\nonumber\\
=&(1+\lambda)^2 \sum\limits_{n=0}^{\infty}\dfrac{q^{n^2+n} \lambda^{2n} (1-\lambda^2 q^{2n+1}) }{(1+\lambda^2 q^{2n+1})}.
\end{align}
We remind that $t$-dependence on the right hand side is implicit in 
$\lambda=\lambda(t)$. To extract the coefficient of $t^{2N}$, we evaluate the contour integral
\begin{equation}
[t^{2N}]G'(t,q)=\frac{1}{2 \pi i}\oint \dfrac{G'(t,q)}{t^{2N+1}}dt.
\end{equation}
Using the variable substitution from \eqref{valuet}, we have
\begin{multline}
[t^{2N}]G'(t,q)=\\\frac{1}{2 \pi i}\oint \left(\dfrac{(1+\lambda)^{2N} \left(\sum\limits_{n=0}^{\infty}\dfrac{q^{n^2+n} \lambda^{2n} (1-\lambda^2 q^{2n+1}) }{(1+\lambda^2 q^{2n+1})}\right)(1-\lambda^2)}{\lambda^{2N}(1-q)^{2N})}\right)\frac{d\lambda}{\lambda}.
\end{multline}
We thus have
\begin{equation}
[t^{2N}]G'\left(t,q\right)=[\lambda^0]H'_N(\lambda,q)
\end{equation}
and to extract the constant term in $\lambda$ on the right hand side we now expand $H'_N(\lambda,q)$ as a Laurent series in $\lambda$. For this we write
\begin{equation}
[\lambda^0]H'_N(\lambda,q)=[\lambda^0]T_1 \sum \limits_{k=0}^{\infty}T_2\;,
\label{constantterm}
\end{equation}
where 
\begin{equation*}
T_1=\dfrac{\left(\lambda+\dfrac{1}{\lambda}\right)^{2N} (1-\lambda^2)}{(1-q)^{2N}}
\end{equation*}
and
\begin{equation}
T_2=\dfrac{q^{n^2+n} \lambda^{2n} (1-\lambda^2 q^{2n+1})}{(1+\lambda^2 q^{2n+1})}.
\end{equation}
We find the series expansions as
\begin{equation}
T_1=\dfrac{\sum\limits_{k=0}^{2N+1}\dfrac{(2N)!(2N-2k+1)\lambda^{2k-2N}}{k!(2N-k+1)!}}{(1-q)^{2N}}
\label{bt1}
\end{equation}
and
\begin{equation}
T_2=q^{n^2+n} \lambda^{2n}\left(2\sum\limits_{l=0}^{\infty}(-1)^l(\lambda^2 q^{2n+1})^l-1\right).
\label{bt2}
\end{equation}
Next we substitute the expression \eqref{bt1} and \eqref{bt2} in \eqref{constantterm}, which after some simplification leads to
\begin{multline}
H'_N(\lambda,q)=\dfrac{1}{(1-q)^{2N}}\\
\left(2\sum\limits_{k=0}^{2N+1}\sum\limits_{n=0}^{\infty}\sum\limits_{l=0}^{\infty}
(-1)^lq^{n^2+n+2nl+l}\dfrac{(2N)!(2N-2k+1)\lambda^{2k-2N+2n+2l}}{k!(2N-k+1)!}\right.\\ 
\left.-\sum\limits_{k=0}^{2N+1}\sum\limits_{n=0}^{\infty}q^{n^2+n} \dfrac{(2N)!(2N-2k+1)\lambda^{2k-2N+2n}}{k!(2N-k+1)!}\right).
\label{bt}
\end{multline}
We aim to get the constant term in $\lambda$, so we combine the powers of $\lambda$ and equate them to $0$. 
This fixes the summation index $k$, and we get
\begin{multline}
[t^{2N}]G'\left(t,q\right)=\dfrac{1}{(1-q)^{2N}}\times\\ \left(2\sum\limits_{n,l=0}^{\infty}
(-1)^lq^{n^2+n+2nl+l}\dfrac{(2N)!(2n+2l+1)}{(N-n-l)!(N+n+l+1)!}\right.\\ 
\left.-\sum\limits_{n=0}^{\infty}q^{n^2+n} \dfrac{(2N)!(2n+1)}{(N-n)!(N+n+1)!}\right).
\end{multline}
Completing the square in the first sum and changing summation indices leads to 
\begin{multline}
[t^{2N}]G'\left(t,q\right)=\dfrac{1}{(1-q)^{2N}}\left(\sum\limits_{m=0}^{N}q^{m^2+m}\dfrac{(2N)!(2m+1)}{(N-m)!(N+m+1)!}\right.\\\left.\left(2\sum\limits_{l=0}^{m}
(-1)^l q^{-l^2}-1\right)\right).
\end{multline}
\end{proof}
\section{Identities}
The central results of this chapter have been given in Theorems \ref{theorem3} and \ref{theoremnot3}, which express finite continued fractions in terms of basic hypergeometric functions. For example, for $q$-tangent numbers we have
\begin{multline*}
\cfrac{1}{1-\cfrac{\alpha^2(1-q)(1-q^2)}{1-\cfrac{\alpha^2(1-q^2)(1-q^3)}{{\ddots-}{\cfrac{\alpha^2(1-q^{w-1})(1-q^{w})}{1-\alpha^2(1-q^w)(1-q^{w+1})}}}}}=\\
\cfrac{1}{1-\dfrac{\lambda^2(1-q)\left[\bar{\lambda}^w\phi(\lambda,q^{3})\phi\left(\bar{\lambda},q^{w+3}\right)-\lambda^w\phi\left(\bar{\lambda},q^{3}\right)\phi(\lambda,q^{w+3})\right]}{{(1+\lambda^2)\left[\bar{\lambda}^w\phi(\lambda,q^{2})\phi\left(\bar{\lambda},q^{w+3}\right)
-\lambda^{w+2}\phi\left(\bar{\lambda},q^{2}\right)\phi(\lambda,q^{w+3})\right]}}}
\end{multline*}
and a similar result holds for $q$-secant numbers. The point we would like to make in this section is that these results can be interpreted as giving hierarchies of identities for basic hypergeometric functions. For $w$ small, the left hand side is a relatively simple rational function in $t$ and $q$, whereas the right hand side is a weighted ratio of products of basic hypergeometric functions at specific arguments. We make the resulting identities explicit for $w=1$ in the following corollary.
\begin{corollary}
\begin{multline}
\tiny
\frac{1-q^2}{1-\nu^2}=\dfrac{\left[ \splitdfrac{\,_2\phi_1(\nu,-\nu;-\nu^2 q;q,q^3)\,_2\phi_1\left(-\bar{\nu},\bar{\nu};-\nu^2\bar{ q};q,q^4\right)}{+\nu\,_2\phi_1\left(-\bar{\nu},\bar{\nu};-\nu^2\bar{ q};q,q^3\right)\,_2\phi_1(\nu,-\nu;-\nu^2 q;q,q^4)}\right]}{\left[\splitdfrac{\,_2\phi_1(\nu,-\nu;-\nu^2 q;q,q^2)\,_2\phi_1\left(-\bar{\nu},\bar{\nu};-\nu^2\bar{ q};q,q^4\right)}{-\nu^4\,_2\phi_1\left(-\bar{\nu},\bar{\nu};-\nu^2 \bar{ q};q,q^2\right)\,_2\phi_1(\nu,-\nu;-\nu^2 q;q,q^4)}\right]}
\end{multline}
where $\nu=i \lambda$, $\bar{\nu}=\frac{1}{\nu}$ and $\bar{q}=\frac{1}{q}$ , and
\begin{multline}
\tiny
\frac{1-q}{1-\mu^2\bar{q}}=\dfrac{\left[\splitdfrac{\,_2\phi_1(\mu,-\mu;-\mu^2;q,q^2)\,_2\phi_1\left(-q\bar{\mu},q\bar{\mu};-q^2\bar{\mu}^2;q,q^3\right)}{+\left(\mu^2\bar{q}\right)\,_2\phi_1\left(-q\bar{\mu},q\bar{\mu};-q^2\bar{\mu}^2;q,q^2\right)\,_2\phi_1(\mu,-\mu;-\mu^2;q,q^3)}\right]}{\left[\splitdfrac{\,_2\phi_1(\mu,-\mu;-\mu^2;q,q)\,_2\phi_1\left(-q \bar{\mu},q\bar{\mu};-q^2\bar{\mu}^2;q,q^3\right)}{-\mu^4\bar q^2\,_2\phi_1\left(-q\bar{\mu},q\bar{\mu};-q^2\bar{\mu}^2;q,q\right)\,_2\phi_1(\mu,-\mu;-\mu^2;q,q^3)}\right]}
\end{multline}
where $\mu=i \lambda \sqrt{q}$, $\bar{\mu}=\frac{1}{\mu}$ and $\bar{q}=\frac{1}{q}$.
\end{corollary}
\begin{proof}
Insert $w=1$ in Theorems \ref{theorem3} and \ref{theoremnot3} and simplify.
\end{proof}
To the best of our knowledge these identities are new. It would be interesting to find an alternative derivation and perhaps deeper understanding of their meaning.

\end{document}